%% file: Stability-KleinGordon-Rev4.tex
\documentclass[12pt,onecolumn,letter]{article}
\input{MyMacros}

\begin{document}
    \title{Stability estimate for the relativistic Schr\"odinger 
           equation with time-dependent vector potentials}
    \author{Salazar, Ricardo. }

    \date{\today}

\maketitle

    \begin{abstract}
        We consider the relativistic Schr\"odinger 
	equation with a time dependent vector and scalar potential on 
	a bounded cylindrical domain. Using a Geometric Optics
	Anzats we stablish a logarithmic stability estimate for the recovery
	of the vector potentials.  
    \end{abstract}

\input{00-Intro}

\input{01-GOAnzats-GreensFormula}

\input{02-Stability}

\input{03-Scalar}


\pagebreak
\input{10-References}

\end{document}

%% file: MyMacros.tex
    \usepackage{amsmath}
    \usepackage{amsfonts}
    \usepackage{amssymb}
    \usepackage{amsthm}
    \usepackage{amstext}
    \usepackage{graphicx}

    \newcommand{\rmi}{\ensuremath{\mathrm{i}}}
    \newcommand{\rme}{\ensuremath{\mathrm{e}}}
    \newcommand{\rmd}{\ensuremath{\mathrm{d}}}
    
    \newcommand{\TT}{\ensuremath{ [T_{1},T_{2}] }}
    \newcommand{\ERRE}[1]{\ensuremath{\mathbb{R}^{#1}}}
    \newcommand{\ELEDOT}[2]{\ensuremath{\left\langle #1,#2 \right\rangle }}
    \newcommand{\ELEDOTT}[2]{\ensuremath{\big\langle #1,#2 \big\rangle_{\TT{}\times\Omega} }}
    \newcommand{\HYPEQN}[3]{\ensuremath{
                            \big( -\rmi\partial_{t} + {#1}_{0}(t,x)\big)^{2}{#3}
                            - \sum_{j=1}^{n} \big(-\rmi\partial_{x_{j}} + {#1}_{j}(t,x) \big)^{2} {#3}
                            + {#2}(t,x){#3} }}
    \newcommand{\HYPOP}[2]{\ensuremath{
                            \big( -\rmi\partial_{t} + {#1}_{0}(t,x)\big)^{2}
                            - \sum_{j=1}^{n} \big(-\rmi\partial_{x_{j}} + {#1}_{j}(t,x) \big)^{2}
                            + {#2}(t,x) }}
    
    \newcommand{\TRANSPOP}[1]{\ensuremath{
                            -\big( \partial_{t} + \rmi {#1}_{0}(t,x)\big)
                            - \sum_{j=1}^{n} \omega_{j}\big( \partial_{x_{j}} + \rmi {#1}_{j}(t,x) \big)}}

    \newcommand{\DT}{\ensuremath{\partial_{t}}}
    
    \newcommand{\DXJ}[1]{\ensuremath{\partial_{x_{#1}}}}

    \newcommand{\POTENTIALS}[2]{\ensuremath{\big(\mathcal{#1} (t,x),#2 (t,x)\big)}}

    \theoremstyle{definition}
    \newtheorem{defn}{Definition}[section]

    \theoremstyle{plain}
    \newtheorem{lemma}{Lemma}[section]

    \theoremstyle{plain}
    \newtheorem{thm}[lemma]{Theorem}

    \theoremstyle{plain}

    \theoremstyle{plain}
    \newtheorem{cor}[lemma]{Corolary}

    \theoremstyle{plain}


%% file: 00-Intro.tex
\section{Introduction}
     Let $\Omega$ be a bounded domain in $\ERRE{n}$, $n\ge 2$,
     consider the hyperbolic equation with time dependent coefficients
     \begin{equation}\label{hyp:eqn}
     \HYPEQN{A}{V}{u} = 0 \quad\text{in }\ERRE{}\times\Omega,
     \end{equation}
     where $A_{j}(t,x)$, $0\le j \le n$, and $V(t,x)$ 
     are compactly supported smooth functions. 

     The vector field $\mathcal{A}(t,x)=
     (A_{0}(t,x),\dots,A_{n}(t,x))$ is called the \emph{vector potential}, the
     function $V(t,x)$ is called the \emph{scalar potential} and 
     equation (\ref{hyp:eqn}) is often referred to as the \emph{relativistic Schr\"odinger
     equation} or, in the case where the vector potential is zero and the scalar
     potential is proportional to the mass of a free particle, it is referred to as
     the \emph{Klein-Gordon equation} (see \cite{Schiff}).
     
     We impose the initial and boundary conditions
    \begin{align}\label{hyp:eqn1a}
     u(t,x) = \partial_{t}u(t,x) & = 0\makebox[2.5 em]{}
                                      \quad\text{for }t << 0 \\ \label{hyp:eqn1b}
                          u(t,x) & = f\makebox[2.5 em]{$(t,x)$}
                                      \quad\text{on }\ERRE{}\times\partial\Omega,
     \end{align}
     where $f$ is a compactly supported smooth function on $\ERRE{}\times\partial\Omega$.
     Solutions to (\ref{hyp:eqn}) satisfying (\ref{hyp:eqn1a}) and (\ref{hyp:eqn1b})
     exist and are unique (Theorem 8.1 in \cite{Isakov}) and we can define the \emph{Dirichlet to Neumann} operator by
     \begin{equation}\label{D:to:N}
     \Lambda(f) := \left( \partial_{\nu} + \rmi A(t,x)\cdot\nu
     \right)u(t,x) \Big|_{\ERRE{}\times\partial\Omega}
     \end{equation}
     where $u$ is the solution of (\ref{hyp:eqn})-(\ref{hyp:eqn1b}), $\nu$ is the
     exterior unit normal to $\partial\Omega$ and we have set $A(t,x) =
     (A_{1}(t,x),\dots,A_{n}(t,x))$.
     The \emph{Inverse Boundary Value Problem} is the recovery of $\mathcal{A}(t,x)$
     and $V(t,x)$ knowing $\Lambda(f)$ for all $f\in C_{0}^{\infty}\big(\ERRE{}
     \times \partial\Omega \big)$.

  \begin{defn}\label{def:gauge:equiv}
     The pair $\POTENTIALS{A}{V}$ and
     $\POTENTIALS{A'}{V'}$ are said to be \emph{gauge equivalent} if
     there exists $g(t,x) \in
     C^{\infty}(\ERRE{}\times\overline{\Omega})$ such that $g(t,x)\not = 0$
     on $\ERRE{}\times\overline{\Omega})$, $g=1$ on $\ERRE{}\times\partial\Omega$ and
     \begin{align*}
        \mathcal{A}'(t,x)= & \mathcal{A}(t,x) - \frac{\rmi}{g(t,x)} \nabla_{t,x} g(t,x) \\
        V'(t,x) = & V(t,x),
     \end{align*}
     where $\nabla_{t,x} :=(\partial_{t},\partial_{x})=
     (\partial_{t},\partial_{x_{1}},\dots,\partial_{x_{n}})$ is the $(n+1)$-dimensional
     gradient. \\
     The mapping $(\mathcal{A},V)\to(\mathcal{A}^{\prime},V^{\prime})$
     is called a \emph{gauge transform}. \\
     The Dirichlet to Neumann maps $\Lambda$ and $\Lambda'$ are said to be 
     \emph{gauge equivalent}
     if for all $f(t,x)\in C_{0}^{\infty}(\ERRE{}\times\partial\Omega)$,
     \begin{equation}\label{eq:1.7b}
        \Lambda'\big(g(t,x)f(t,x)\big) =
        g(t,x) \Lambda\big(f(t,x) \big).
     \end{equation}
  \end{defn}
  \textbf{Remark:} When $\Omega$ is simply connected, the
  gauge $g$ has the particular form $g(t,x)=\rme^{\rmi\varphi(t,x)}$ where
  $\varphi(t,x)\in C^{\infty}(\ERRE{}\times\Omega)$.
  Then $-\frac{\rmi}{g(t,x)}\nabla_{(t,x)}g(t,x)=\nabla_{(t,x)}\varphi(t,x)$
  and two vector potentials are gauge equivalent if their
  difference is the gradient of a smooth function.
  \\

      Inverse problems is a topic in mathematics that has been growing in interest 
      in part, due to its wide range of applications, from medicine
      to acoustics to electromagnetism (see for instance \cite{Isakov}
      for some of the latest tools and techniques employed in the solutions of these problems).
      In the case of the hyperbolic inverse boundary value problem (1)-(4) with time 
      independent coefficients, a powerful tool called the \emph{boundary control method}, 
      or BC-method for short, was discovered by Belishev (see \cite{Belishev}). It was later 
      developed by Belishev, Kurylev, Lassas, and others (\cite{Kurylev},\cite{Kurylev:Lassas}). 
      A new approach to this problem based on the BC-method was developed by Eskin in 
      (\cite{Eskin:approach},\cite{Eskin:approach2}). 
      On a similar note, Stefanov and Uhlmann 
      established stability results for the wave equation in anisotropic 
      media (see \cite{Stefanov:Uhlmann}, \cite{Stefanov:Uhlmann:2} and 
      \cite{Uhlmann} for a survey of these results). 

      In the case where the scalar potential is time-dependent and the vector potential
      is identically equal to zero ($\mathcal{A}\equiv 0$ in (1)), Stefanov 
      \cite{Stefanov} and Ramm-Sj\"ostrand \cite{Ramm:Sjostrand}, have 
      shown that the Dirichlet to Neumann map completely determines the
      scalar potentials. In \cite{Eskin:time:dependent}, Eskin
      considered the case with time-dependent potentials that are analytic in time 
      (this case is more general in terms of the complexity of the PDE but less general 
      with its assumption of analiticity). 
      The analiticity of the time variable is related to the use 
      of a unique continuation theorem established by Tataru in \cite{Tataru}. 
      More recently, the results of \cite{Ramm:Sjostrand,Stefanov} were generalized
      by the author in \cite{Salazar,Salazar2} for the case of vector potentials,
      where it was shown that the Dirichlet-to-Neumann operator determines the 
      vector and scalar potentials up to a gauge transform. 

      Regarding elliptic problems, 
      the questions of stability and reconstruction have been studied for several 
      IBVP 
      (see \cite{Alessandrini,Choulli,DosSantos:Uhlmann,DosSantos:Kenig:Salo,Isakov,Uhlmann} 
      and the references therein).
      For the parabolic case, there are a few results concerning the determination
      of time-dependent coefficients in an IVBP. The case of a source term of 
      the form $f(t)\chi_{D}$, where $\chi_{D}$ is the charachteristic function of
      a known subdomain was considered by Perez-Esteva and Canon in \cite{Chavita} in a
      half line in one dimension. Later in \cite{Chavita2} they considered a similar 
      problem in 3 dimensions. 
      For more references on recent developments on uniqueness and stability estimates
      on elliptic and parabolic PDE's the reader is referred to the books by Isakov \cite{Isakov}
      and Choulli \cite{Choulli}.

      Regarding the stability in the hyperbolic case, the first results were obtained 
      by Isakov in \cite{Isakov:first}.
      Isakov and Sun \cite{Isakov:Sun} obtained estimates for two coefficients of a hyperbolic 
      partial differential equation from all measurements on a part of the lateral boundary. 
      In \cite{Stefanov:Uhlmann} Stefanov and Uhlmann studied the hyperbolic Dirichlet to 
      Neumann map associated to the wave equation in 
      anisotropic media; and in \cite{Stefanov:Uhlmann:2}, they consider the 
      more general case of determining a Riemannian metric on a Riemannian
      manifold with boundary from the boundary measurements. More recently
      in \cite{Montalto}, Montalto recovers the metric, a covector field and 
      a potential from the hyperbolic Dirichlet to Neumann map.

      However, stability in the case of time-dependent vector has not been 
      considered before. In this paper, which is based on \cite{Salazar2,Salazar3}, 
      we take advantage of a result by Begmatov \cite{Begmatov},
      where he proves a stability estimate for a time-dependent scalar 
      function when information about its X-ray transforms is known on a cone. 
      In our work we establish stability estimates for vector and scalar potentials
      when they are compactly supported in space and time. 
      This work is structured as follows. 
    In section \ref{sec:review:prev:article} we review the 
    construction of the Geometric Optics Anzats (GO for short) as well as the Green's 
    formula developed in \cite{Salazar}. This construction is later used to 
    obtain estimates for the X-ray transform along `light rays' of particular combinations 
    of the components of the vector potentials. In section \ref{sec:stability:estimates} a 
    logarithmic stability estimate for vector potentials is established, and finally 
    in section \ref{new:section} 
    we prove an estimate for the case when both vector 
    and scalar potentials are present.

%% file: 01-GOAnzats-GreensFormula.tex
\section{Geometric optics and Green's formula}\label{sec:review:prev:article}

  The following Geometric Optics construction is the same in
  \cite{Salazar}, however, it is included here because some of the details
  will be needed in the estimates in section \ref{sec:stability:estimates}.

  For the hyperbolic problem (\ref{hyp:eqn})-(\ref{hyp:eqn1b}) 
  Geometric Optics Anzats supported near light rays take the form
  \begin{equation}\label{geom:optic:sols}
     u(t,x) = \rme^{\rmi k(t-\omega\cdot x)}
              \sum_{p=0}^{N}\frac{v_{p}(t,x)}{(2\rmi k)^{p}}
              + v^{(N+1)}(t,x),\qquad \omega\in S^{n-1}, k\in\mathbb{R}.
  \end{equation}
  For $u$ as above, equation (\ref{hyp:eqn}) yields (see \cite{Salazar} 
  for the details)
  \begin{equation}\label{hyp:eqn:for:u1}
    0= \big(L +2\rmi k\mathcal{L} \big) v, 
  \end{equation}
  where 
  \begin{align}\label{geom:series:soln}
     v(t,x) & = \sum_{p=0}^{N} \frac{v_{p}(t,x)}{(2ik)^{p}} + \rme^{-\rmi k(t-\omega\cdot x)}v^{(N+1)}(t,x) \\
     L & = \HYPOP{A}{V} \\
     \mathcal{L} & = \TRANSPOP{A},
  \end{align}
   plugging in $v$ into (\ref{hyp:eqn:for:u1}) then gives 
  \begin{align}\nonumber
     (2\rmi k)\mathcal{L}v_{0} & + \big( \mathcal{L}v_{1} + L v_{0} \big) +
     \frac{1}{(2\rmi k)}\big( \mathcal{L}v_{2} + L v_{1} \big) + \dots + \\
     & \frac{1}{(2\rmi k)^{N-1}}\big( \mathcal{L}v_{N} + L v_{N-1} \big) +
     \frac{1}{(2\rmi k)^{N}}L v_{N} + \rme^{-\rmi k(t-\omega\cdot x)}L v^{(N+1)} = 0.
  \end{align}
    To ensure that the previous equation is satisfied, we can use 
    a two-step process. In the first step we solve the $N+1$ transport 
    equations
    \begin{equation}
      \mathcal{L} v_{0} = 0,\qquad \mathcal{L} v_{j}
                        = -L v_{j-1}, \qquad 1\le j \le N
    \end{equation}
    with initial conditions supported in a small neighborhood of a point 
    $(t,x) \in \ERRE{}\times\partial\Omega$,
    and in the second step we solve the second order equation
    \begin{equation}\label{hyp:eqn:after:go:solns}
    L v^{(N+1)} = -\frac{\rme^{\rmi k(t-\omega\cdot x)}}{(2\rmi k)^{N}} L v_{N}
    \end{equation}
    with initial and boundary conditions
    \begin{align*}
     v^{(N+1)}(t,x) & = 0
                    & \text{for }t = T_{1} \quad \phantom{nx\in\partial\Omega}\\
     \partial_{t}v^{(N+1)}(t,x) & = 0
                    & \text{for }t = T_{1} \quad \phantom{nx\in \partial\Omega}\\
     v^{(N+1)}(t,x) & = 0
                    & \text{for }t \ge T_{1}, \quad x\in\partial\Omega.
     \end{align*}

    This differential equation admits a unique solution; moreover if we denote
    by $h$ the right hand side of (\ref{hyp:eqn:after:go:solns}), 
    then for $T_{1}<t<T$ and $k>1$ (see \cite{Isakov}, pp. 185) 
    \begin{align}\nonumber
     ||\partial_{t}v^{(N+1)}(t,\thinspace\cdot)||_{L^{2}(\Omega)} +
     ||v^{(N+1)}(t,\thinspace\cdot)||_{H^{1}(\Omega)}
       & \le C
       ||h||_{L^{2}((T_{1},T)\times\Omega)} \\ \label{power:of:k}
       & \le \frac{C}{k^{N}} .
    \end{align}
    If $v_{0}$ is a solution of the transport equation
    $\mathcal{L}v_{0}(t,x)=0$, it has the form
    \begin{equation}\label{eqn:2.24}
       v_{0}(t,x)=\chi_{1}(t',x')
       \exp \Big[ -\rmi
          \int_{-\infty}^{(t+\omega\cdot x)/2} 
	       \sum_{j=0}^{n} \omega_{j}A_{j}(t^{\prime}+s,x^{\prime}+s\omega)
	       \thinspace\mathrm{d}s
       \Big]
    \end{equation}
    where $(t^{\prime},x^{\prime})=(t,x)-\frac{1}{2}(t+\omega\cdot x)(1,\omega)$ is
    the projection of $(t,x)$ into $\Pi_{(1,\omega)}$, the $n$-dimensional 
    linear subspace perpendicular to $(1,\omega)$ (see figure 1 in \cite{Salazar}), 
    and $\chi_{1}$ is any real valued function that is constant along the direction given 
    by $(1,\omega)$, and whose support is contained in a neighborhood of the light ray 
    $\gamma=\{ (t',x') + s(1,w)\thinspace | \thinspace s\in\ERRE{}\}$. 
 
    It then follows that 
    $u=\rme^{\rmi k(t-\omega\cdot x)}(v_{0}+\mathcal{O}(k^{-1}))$
    solves (\ref{hyp:eqn}) and satisfies the set of initial conditions
    (\ref{hyp:eqn1a}). 
    Summarizing, a GO solution of (\ref{hyp:eqn})-(\ref{hyp:eqn1a}) of the form
    \begin{equation}\label{GO1a}
    u(t,x) = \exp\big[\rmi k(t-\omega\cdot x)- \rmi R_{1}(t,x;\omega)\big]
             \big( \chi_{1}(t',x') + \mathcal{O}(k^{-1}) \big)
    \end{equation}
    can be constructed, where
    \begin{equation}\label{GO1b}
    R_{1}(t,x;\omega) = \int^{(t+\omega\cdot x)/2}_{-\infty}
                 \sum_{j=0}^{n}\omega_{j}A_{j}(t^{\prime}+s,x^{\prime}+s\omega)
             \thinspace\mathrm{d}s.
    \end{equation}
    Similarly, a GO solution for the backwards hyperbolic
    problem can be obtained in the same fashion, with another  
    real valued function $\chi_{2}$ constant along a given light ray.
    
   To obtain a Green's formula for this problem, we let $T_{1}$ and $T_{2}$ be two real 
   numbers with $T_{1}<<0<<T_{2}$, and consider the forward and backward hyperbolic equations
   \begin{center}
   \begin{tabular}{r l c r l}
    $L_{1} u = 0$ &$\quad\text{in }\TT{}\times\Omega$ & $\qquad$&
    $L_{2}^{*} v = 0$ &$\quad\text{in }\TT{}\times\Omega$ \\
    $u = \DT u = 0$ &$\quad\text{for } t = T_{1} $ & $\quad$&
    $v = \DT v = 0$ &$\quad\text{for } t = T_{2} $ \\
    $u=f$           &$\quad\text{on }\TT{}\times\partial\Omega$ & $\quad$ &
    $v=g$           &$\quad\text{on }\TT{}\times\partial\Omega$,
   \end{tabular}
   \end{center}
   where 
   \begin{align*}
      L_{1} 
      & = \big( -i\DT + A_{0}^{(1)}(t,x)\big)^{2} -
          \sum_{j=1}^{n} \big(-i\DXJ{j} + A_{j}^{(1)}(t,x) \big)^{2} +
          V^{(1)}(t,x) \\
      L_{2}^{*} 
      & = \big( -i\DT + \overline{A_{0}^{(2)}(t,x)}\big)^{2} -
          \sum_{j=1}^{n} \big(-i\DXJ{j} + \overline{A_{j}^{(2)}(t,x)} \big)^{2} +
          \overline{V^{(2)}(t,x)}.
   \end{align*}
   If we denote by $\langle\text{ , }\rangle_{\TT{}\times\Omega}$ and
   $\langle\text{ , }\rangle_{\TT{}\times\partial\Omega}$ the $L^{2}$ 
   inner products in $\TT{}\times\Omega$, $\TT{}\times\partial\Omega$;  
   integration by parts applied to 
   \begin{equation*}
      \ELEDOTT{L_{1}u}{v} - \ELEDOTT{u}{L_{2}^{*}v} = 0
   \end{equation*}
   yields the Green's Formula (see \cite{Salazar} for complete details)
   \begin{multline}\label{eqn:GF}
           \ELEDOT{\Lambda_{1}\big(f\big)}{g}_{\TT{}\times\partial\Omega}
         - \ELEDOT{\Lambda_{2}\big(f\big)}{g}_{\TT{}\times\partial\Omega}
         = \\
           \sum_{j=0}^{n} r_{j} \Big(
           \ELEDOTT{A_{j}u}{(-\rmi\DXJ{j} v)} 
	  + \ELEDOTT{A_{j}(-\rmi\DXJ{j} u)}{v}
          \Big) \\ 
           + \sum_{j=0}^{n} r_{j}
           \ELEDOTT{\big( (A_{j}^{(2)} )^{2}
          - (A_{j}^{(1)} )^{2}\big)u}{v} 
       - \ELEDOTT{Vu}{v},
   \end{multline}
   where $x_{0}=t$, $A_{j}=A_{j}^{(2)}-A_{j}^{(1)}$ for
   $0\le j\le n$, $V=V^{(2)}-V^{(1)}$, $r_{0}=-1$, and $r_{j}=1$ for $1\le j\le n$.

%% file: 02-Stability.tex
\section{Stability of the vector potentials}\label{sec:stability:estimates}

    The proof in this section closely follows \cite{Salazar2}. We 
    assume that the components of the vector potentials
    $\mathcal{A}^{(1)}$ and $\mathcal{A}^{(2)}$ as well as the 
    scalar potentials $V^{(1)}$ and $V^{(2)}$ are real valued, 
    smooth and compactly supported in both $t$ and $x$. We write
    \begin{equation*}
       \mathcal{A} = \mathcal{A}^{(1)} - \mathcal{A}^{(2)} \quad\text{where}\quad
       \mathcal{A}^{(k)} = (A_{0}^{(k)},\dots,A_{n}^{(k)}),\quad k=1,2,
    \end{equation*}
    and as before we denote by $\Pi_{(1,\omega)}$ the
    $n$-dimensional linear subspace perpendicular to $(1,\omega)$. In symbols
    $$\Pi_{(1,\omega)} = \{ (t,x)\, : \, t + \omega\cdot x = 0 \}.$$

   The GO anzats and the Green's formula developed
   in the previous section allow for the estimation
   of the X-ray transform over light rays of particular combinations of the 
   components of the vector potentials. The precise statement is as follows:
   \begin{lemma}\label{lemma:4.1}
      If $\Lambda_{k}$, $k=1,2$ 
      represents the Dirichlet to Neumann operator for the hyperbolic equations
      \begin{equation}\label{hyp:eqn:no:scalar:pot}
      \Big(\big( -\rmi\DT + A^{(k)}_{0}(t,x)\big)^{2}
                            - \sum_{j=1}^{n} \big(-\rmi\DXJ{j} + A^{(k)}_{j}(t,x) \big)^{2}
                          + V^{(k)}(t,x)  \Big) u = 0, 
      \end{equation}
      then for all $(t,x) \in \mathbb{R}^{n+1}_{t,x}, \omega\in S^{n-1},$ 
      the vectorial ray transform of $\mathcal{A} = (A^{(2)}_{0}-A^{(1)}_{0},\dots,
      A^{(2)}_{n}-A^{(1)}_{n})$ along the light rays 
      \begin{equation*}
      \gamma_{(t,x;\omega)}
      =\{ (t,x)+s(1,\omega)\thinspace : s\in\mathbb{R} \},
      \end{equation*}
      satisfies
    \begin{equation}\label{eqn:58}
       \Bigg|
       \exp \Big[ 
              i \int_{-\infty}^{\infty} 
           \big( A_{0} + \sum_{j=1}^{n} \omega_{j}A_{j}
           \big) ( t + s, x + s\omega )
             \thinspace \mathrm{d} s 
            \Big] - 1
       \Bigg| \le
       C\, ||| \Lambda_{1} - \Lambda_{2} |||,
    \end{equation}
      where $|||\phantom{m}|||$ represents the operator norm
      between $H^{1}([T_{1},T_{2}]\times\partial\Omega)$ and
      $ L^{2}([T_{1},T_{2}]\times\partial\Omega)$. 
   \end{lemma}
   \textbf{Remark: } We point that this result is independent of the presence of scalar 
   potentials. 
   \begin{proof}
      Owing to (\ref{GO1a}) and (\ref{GO1b}), GO Anzats for
      the forward and backward hyperbolic equations
      are given by
      \begin{align}\label{geom:optic:1a}
      u(t,x) & = \exp\Big[\rmi k(t-\omega\cdot x)- \rmi R_{1}(t,x;\omega) \Big]
               \Big( \chi_{1} + \mathcal{O}\left(k^{-1}\right) \Big), \\ \label{geom:optic:1b}
      \overline{v(t,x)} & =
                 \exp\Big[-\rmi k(t-\omega\cdot x) + \rmi \overline{R_{2}(t,x;\omega)}\Big]
               \Big( \chi_{2} + \mathcal{O}\left(k^{-1}\right) \Big),
      \end{align}
      where
      \begin{align}\label{geom:op:r1}
         R_{1}(t,x;\omega)
            & = \int^{(t+\omega\cdot x)/2}_{-\infty}
              \sum_{j=0}^{n} \omega_{j}
              A^{(1)}_{j}(t^{\prime}+s,x^{\prime}+s\omega)
              \thinspace \mathrm{d}s, \\ \label{geom:op:r2}
              \overline{R_{2}(t,x;\omega)}
            & = \int^{(t+\omega\cdot x)/2}_{-\infty}
              \sum_{j=0}^{n} \omega_{j}
              A^{(2)}_{j}(t^{\prime}+s,x^{\prime}+s\omega)
              \thinspace \mathrm{d}s,
      \end{align}
      where $\chi_{1},\chi_{2}$ are constant along, and supported on
      a small neighborhood of the light ray $\gamma_{(t,x;\omega)}$, and where
       $(t',x')$ is the projection of $(t,x)$ into $\Pi_{(1,\omega)}$.

      For $0\le j\le n$, differentiation of \eqref{geom:optic:sols} with
      respect to $x_{j}$ combined with estimate \eqref{power:of:k}
      lead to
      \begin{equation}\label{eqn:derivatives:GOanzats}
         \DXJ{j} u
         = k\exp\Big[\rmi k(t-\omega\cdot x) - \rmi R_{1}(t,x;\omega)\Big]
                \Big(-\rmi r_{j}\omega_{j} \chi_{1} + \mathcal{O}(k^{-1})
                \Big),
      \end{equation}
      where $x_{0}=t$, $\omega_{0}=1$, $r_{0}=-1$ and $r_{j}=1$ when $j\neq 0$.
      Then by (\ref{geom:optic:1b})
      \begin{equation}\label{GO:product:v:du}
         \left(-\rmi\DXJ{j}u(t,x)\right)\overline{v(t,x)}
         = -k\rme^{\rmi\left(\overline{R_{2}(t,x;\omega)}-R_{1}(t,x;\omega)\right)}
           \left( r_{j}\omega_{j}\chi_{1}\chi_{2} + \mathcal{O}(k^{-1}) \right).
      \end{equation}
      Similarly, \eqref{geom:optic:1a} yields
      \begin{equation}\label{GO:product:u:dv}
        u(t,x)\overline{\left(-\rmi\DXJ{j}v(t,x)\right)}
        = -k\rme^{\rmi\left(\overline{R_{2}(t,x;\omega)}-R_{1}(t,x;\omega)\right)}
          \left( r_{j}\omega_{j}\chi_{1}\chi_{2} + \mathcal{O}(k^{-1}) \right).
      \end{equation}

      Denoting by $\mathcal{I}_{\mathrm{R}}$ the right hand side of \eqref{eqn:GF},
      we obtain via the previous two formulas
    \begin{multline*}
         \mathcal{I}_{\mathrm{R}}
          =Ck \int_{T_{1}}^{T_{2}}\int_{\Omega} \sum_{j=0}^{n}
          \Big( A_{0} + \sum_{j=1}^{n} \omega_{j}A_{j}\Big)(t,x) \chi_{1}(t,x)\chi_{2}(t,x) 
          \thickspace \times \\ 
          \exp\Big[\rmi\big(\overline{R_{2}(t,x;\omega)} - R_{1}(t,x;\omega)\big)\Big] 
          \thinspace \rmd x\thinspace \rmd t + \cdots
      \end{multline*}
      which in turn leads to
      \begin{multline}\label{eq:GF:RHS}
         \mathcal{I}_{\mathrm{R}}
           = Ck \int_{T_{1}}^{T_{2}}\int_{\Omega} 
          \Big( A_{0} + \sum_{j=1}^{n} \omega_{j}A_{j}\Big)(t,x) \chi_{1}(t,x)\chi_{2}(t,x) 
                 \thinspace \times \\
                 \rme^{-i \int_{-\infty}^{\frac{1}{2} ( t+\omega\cdot x )} 
                        \big( A_{0} + \sum_{j=1}^{n} \omega_{j}A_{j}\big)(t'+s,x'+s\omega) \thinspace \mathrm{d}s } 
                        \thinspace \mathrm{d}x\thinspace \mathrm{d}t
          + \cdots 
      \end{multline}
      where $C$ is a constant and ``$\cdots$'' represents terms of order $\mathcal{O}(1)$.

      We turn now our attention to the left hand side of
      \eqref{eqn:GF}. Denoting by $f$ and $g$ the restrictions of $u$ and $v$ to
      $[T_{1},T_{2}]\times\partial\Omega$,
      that is
    \begin{equation*}
       f = u(t,x) \big|_{[-T_{1},T_{2}]\times\partial\Omega} \qquad
       g = v(t,x) \big|_{[-T_{1},T_{2}]\times\partial\Omega},
    \end{equation*}
    we have by the Cauchy-Schwarz inequality
    \begin{multline*}
       | \mathcal{I}_{\mathrm{R}} | =
       \big| \ELEDOT{(\Lambda_{1} - \Lambda_{2})(f)}{g}_{[T_{1},T_{2}]\times\partial\Omega} \big| \le
       ||| \Lambda_{1} - \Lambda_{2} ||| \times \\
       \thinspace\thinspace || f ||_{H^{1}([T_{1},T_{2}]\times\partial\Omega)} 
                 \thinspace || g ||_{L^{2}([T_{1},T_{2}]\times\partial\Omega)}.
    \end{multline*}
    Using \eqref{geom:optic:1a} the latter norm can be estimated by
    \begin{align}\nonumber 
       || \thinspace g \thinspace ||_{L^{2}([T_{1},T_{2}]\times\partial\Omega)}
       & =   || \thinspace \chi_{2}(t,x) 
              (1+\mathcal{O}(k^{-1})) \thinspace ||_{L^{2}([T_{1},T_{2}]\times\partial\Omega)} \\ \label{norm:g:L2}
       & \le || \thinspace \chi_{2}(t,x) 
       \thinspace ||_{L^{2}([T_{1},T_{2}]\times\partial\Omega)}
              + \mathcal{O}(k^{-1}),
    \end{align}
    whereas by \eqref{eqn:derivatives:GOanzats} the middle norm can be estimated by
    \begin{align}\nonumber
       || \thinspace f \thinspace ||_{H^{1}([T_{1},T_{2}]\times\partial\Omega)}
       & \le C \Big[
           k \big|\big|\thinspace \chi_{1} 
                \thinspace\big|\big|_{L^{2}([T_{1},T_{2}]\times\partial\Omega)}
         + \mathcal{O}(1) \Big]  \\ \label{norm:f:H1}
       & = Ck \Big[
           \big|\big|\thinspace \chi_{1} 
                \thinspace\big|\big|_{L^{2}([T_{1},T_{2}]\times\partial\Omega)}
           + \mathcal{O}(k^{-1}) \Big].
    \end{align}
    In addition, since $\Omega$ is bounded and $\chi_{j}$, $j=1,2$, is localized
    near a light ray, we have
    $ ||\chi_{j}||_{L^{2}([T_{1},T_{2}]\times\partial\Omega)} \le C $.
    Therefore, by (\ref{norm:g:L2}) and (\ref{norm:f:H1})
    \begin{equation} \label{stab:estim:1}
       | \mathcal{I}_{\mathrm{R}} | \le
       Ck \, \Big[ |||\Lambda_{1}-\Lambda_{2} ||| 
       + \mathcal{O}(k^{-1}) \Big].
    \end{equation}

   Dividing both sides of Green's formula \eqref{eqn:GF} by $k$
   (i.e., \eqref{eq:GF:RHS} and \eqref{stab:estim:1}) and taking the
   limit as $k\to\infty$, we obtain via the triangle inequality and
   the change of coordinates $(t,x) = \sigma(1,\omega) + Y^{\prime}$,
   $Y^{\prime} \in \Pi_{(1,\omega)}$
   \begin{multline}\label{eqn:57}
       \Bigg| \int_{\Pi_{(1,\omega)}} 
       \int_{\mathbb{R}}
          \Big( A_{0} + \sum_{j=1}^{n} \omega_{j}A_{j}\Big)\big(Y^{\prime} + \sigma(1,\omega)\big) 
                \chi_{1}(Y^{\prime})\chi_{2}(Y^{\prime}) \thickspace \times \\
                 \rme^{-i \int_{-\infty}^{\sigma} 
                        \big( A_{0} + \sum_{j=1}^{n} \omega_{j}A_{j}\big)
                              \big(Y^{\prime} + s(1,\omega) \big) \thinspace \mathrm{d}s } 
                        \thinspace \mathrm{d}\sigma\thinspace \mathrm{d}S_{Y^{\prime}}
       \Bigg| \le
       C\, |||\Lambda_{1}-\Lambda_{2} |||. 
    \end{multline}

    If we set
    \begin{equation*}
    a(Y^{\prime}) := 
       \int_{\mathbb{R}}
          \Big( A_{0} + \sum_{j=1}^{n} \omega_{j}A_{j}\Big)\big(Y^{\prime} + \sigma(1,\omega)\big) 
                 \thinspace \rme^{-i \int_{-\infty}^{\sigma} 
                        \big( A_{0} + \sum_{j=1}^{n} \omega_{j}A_{j}\big)
                        \big(Y^{\prime} + s(1,\omega) \big) \thinspace \mathrm{d}s } 
                        \thinspace \mathrm{d}\sigma,
    \end{equation*}
    equation (\ref{eqn:57}) can be rewritten as
    \begin{equation*}
       \Big| \int_{\Pi_{(1,\omega)}} a(Y^{\prime}) \chi_{1}(Y^{\prime})\chi_{2}(Y^{\prime}) 
            \thinspace \mathrm{d}S_{Y^{\prime}} \Big| 
       \le
       C |||\Lambda_{1}-\Lambda_{2} |||.
    \end{equation*}
    The conditions imposed on the support of $\chi_{j}$, $j=1,2$, guarantee that the above estimate holds
    for any $\chi_{j}$ satisfying 
    $\int_{\Pi_{(1,\omega)}} |\chi_{j}(Y^{\prime})|^{2}\mathrm{d}S_{Y^{\prime}} \le 1$,
    thus $a$ is a bounded linear functional on $L^{1}(\Pi_{(1,\omega)})$ and the estimate
    \begin{multline*}
       \Bigg| \int_{-\infty}^{\infty} 
           \big( A_{0} + \sum_{j=1}^{n} \omega_{j}A_{j}
          \big) \big( X^{\prime} + \sigma (1,\omega) \big)
             \thickspace \times \\
            \rme^{i \int_{-\infty}^{\sigma} 
                  ( A_{0} + \sum_{j=1}^{n} \omega_{j}A_{j} )
                 ( X^{\prime} + s (1,\omega))
                \thinspace \mathrm{d} s
               }
             \thinspace \mathrm{d} \sigma
       \Bigg| \le
       C ||| \Lambda_{1} - \Lambda_{2} |||
    \end{multline*}
    holds. To finish the proof, we invoke the Fundamental Theorem of Calculus
    and rewrite the integral in the original coordinate system to obtain
    \begin{equation*}
       \Bigg|
       \exp \Big[ 
             i \int_{-\infty}^{\infty} 
           \big( A_{0} + \sum_{j=1}^{n} \omega_{j}A_{j}
          \big) ( t + s, x + s\omega )
             \thinspace \mathrm{d} s 
            \Big] - 1
       \Bigg| \le
       C ||| \Lambda_{1} - \Lambda_{2} |||.
    \end{equation*}
   \end{proof}
   \begin{cor}\label{cor:4.2}
      Let $\Lambda_{1}$, $\Lambda_{2}$, 
      represent the Dirichlet to Neumann operators for the hyperbolic equations
      \eqref{hyp:eqn:no:scalar:pot}, and let
      \begin{equation*}
          \alpha := 
 	  \mathrm{sup}
 	      \Big| \int_{-\infty}^{\infty} \big( A_{0} + \sum_{j=0}^{n} \omega_{j} A_{j} \big) 
 	        (t+s,x+s\omega) \thinspace\mathrm{d}s \Big|
      \end{equation*}
      where the supremum is taken over  
      $(t,x,\omega)\in[T_{1},T_{2}]\times\Omega\times S^{n-1}$. If $\alpha<2\pi$,
      then for all $(t,x) \in \mathbb{R}^{n+1}_{t,x}, \omega\in S^{n-1}.$ 
    \begin{equation}\label{estim:abs:ray:integrals}
       \Big|
           \int_{-\infty}^{\infty} 
           \big( A_{0} + \sum_{j=1}^{n} \omega_{j}A_{j}
           \big) ( t + s, x + s\omega )
             \thinspace \mathrm{d} s 
       \Big| \le
       C\, ||| \Lambda_{1} - \Lambda_{2} |||,
    \end{equation}
      where $|||\phantom{m}|||$ represents the operator norm
      between $H^{1}([T_{1},T_{2}]\times\partial\Omega)$ 
      and $L^{2}([T_{1},T_{2}]\times\partial\Omega)$.
   \end{cor}
   \begin{proof}
     Denoting by $\beta$ the integral 
     $\int_{-\infty}^{\infty} ( A_{0} + \sum_{j=1}^{n} \omega_{j} A_{j} )  
     (t+s,x+s\omega) \thinspace\mathrm{d}s $, 
     we have 
     \begin{equation} \label{eqn:59}
        \frac{\big| \rme^{i\beta} -1 \big|}{|\beta|} = 
	\frac{| \sin \frac{\beta}{2} |}{ \frac{|\beta|}{2}}.
     \end{equation}
     Since $\dfrac{|\beta|}{2} < \dfrac{\alpha}{2} <\pi $, 
     the right hand side of (\ref{eqn:59}) is bounded from below.
     It then follows that 
    \begin{equation*}
       \Big| \int_{-\infty}^{\infty} 
           ( A_{0} + \sum_{j=1}^{n} \omega_{j}A_{j} ) 
	   ( t + s , x + s\omega )
             \thinspace \mathrm{d} s 
	     \Big| \le C
       \Big| \rme^{ i \int_{-\infty}^{\infty} 
           ( A_{0} + \sum_{j=1}^{n} \omega_{j}A_{j} ) 
	   ( t + s , x + s\omega )
             \thinspace \mathrm{d} s }
             - 1 \Big|,
    \end{equation*}
    which in turn leads to \eqref{estim:abs:ray:integrals}.
   \end{proof}

    To deal with the fact that uniqueness of the vector potentials is 
    expected only up to a gauge transform we impose the divergence 
    condition
    \begin{equation}\label{condition:divergence}
       \mathrm{div}\thinspace\mathcal{A} = 
       \partial_{t}A_{0}(t,x) + \sum_{j=1}^{n} \partial_{x_{j}}A_{j}(t,x) = 0.
    \end{equation}
    By the remark after the definition of gauge equivalent pairs of potentials,
    we know that the difference of vector potentials is a the gradient of a 
    scalar function. The divergence condition then implies that said scalar 
    function must also be harmonic and hence equal to zero by the support 
    conditions imposed on the vector potentials.

    Denoting by $F$ the ray transform of 
    $A_{0} + \sum_{j=1}^{n} \omega_{j} A_{j}$ along 
    light rays $\gamma(t,x;\omega)$, we can rewrite 
    (\ref{estim:abs:ray:integrals}) as
    \begin{equation}\label{estimate:F}
       |F(t,x;\omega)| \le C \,||| \Lambda_{1} - \Lambda_{2} |||
    \end{equation}
    for all $(t,x)\in\ERRE{}_{t}\times\ERRE{n}_{x}$, $\omega\in S^{n-1}$.
    Taking the Fourier transform of $F$ in the variables $x_{1},\dots,x_{n}$ 
    yields 
    \begin{equation*}
     \big(\mathcal{F}_{(x\to\xi)}F(t,\cdot;\omega)\big)(\xi) = 
          \int_{\ERRE{n}} 
          \rme^{-i\xi\cdot x} 
	  \int_{\ERRE{}} \Big(A_{0} + \sum_{j=1}^{n} \omega_{j} A_{j} \Big) 
	                 (t+s,x+s\omega) \thinspace\mathrm{d}s
	              \thinspace \mathrm{d}x,
    \end{equation*}
    and the change of coordinates $\tilde{x}=x+s\omega$, $\tilde{t}=t+s$, with
    Jacobian 
     $\big| \frac{ \partial (\tilde{t},\tilde{x}) }{ \partial (t,x)} \big| = 1$
     leads to
    \begin{align*}
     \big(\mathcal{F}_{(x\to\xi)}F(t,\cdot;\omega)\big)(\xi)
       & = \rme^{-i(\omega\cdot\xi)t} 
          \int_{\ERRE{n}} 
	  \int_{\ERRE{}} 
	                 \rme^{-i\tilde{x}\cdot\xi}
			 \thinspace \rme^{-i(-\omega\cdot\xi)\tilde{t}} 
	                 \Big(A_{0} + \sum_{j=1}^{n} \omega_{j} A_{j} \Big) 
	                 (\tilde{t},\tilde{x}) \thinspace\mathrm{d}\tilde{t}
          \thickspace 
	              \mathrm{d}\tilde{x},
    \end{align*}
    where the right hand side of the above equation is the Fourier
    transform (in all variables) of $A_{0}+\sum_{j=1}^{n}\omega_{j}A_{j}$
    at the point $(-\omega\cdot\xi,\xi)$. The above equation can be rewritten
    as
    \begin{equation*}
	\rme^{it\omega\cdot\xi}
        \big(\mathcal{F}_{(x\to\xi)}F(t,\cdot;\omega)\big)(\xi) =
	\Big(A_{0} + \sum_{j=1}^{n} \omega_{j} A_{j} \Big)^{\wedge} (-\omega\cdot\xi,\xi) 
    \end{equation*}
    and since the right hand side is independent of
    $t$, so is the left hand side. In particular when $t=0$ we have
    \begin{equation}\label{Fourier:trans:indep:t}
	\Big(A_{0} + \sum_{j=1}^{n} \omega_{j} A_{j} \Big)^{\wedge} (-\omega\cdot\xi,\xi) =
        \big(\mathcal{F}_{(x\to\xi)}F(0,\cdot;\omega)\big)(\xi) =:
	G(\xi;\omega).
    \end{equation}
    Since the potentials $A_{j}$ are smooth and compactly supported,
    $F(0,\cdot;\cdot):\ERRE{n}_{x}\times S^{n-1}\to \ERRE{}$
    is also smooth and compactly supported because 
    for $|x|$ big enough, the light rays with direction $(1,\omega)$ 
    emanating from the point $(0,x)$ do not intersect the support of
    the potentials $A_{j}$. 
    Moreover by (\ref{estimate:F}) it is uniformly bounded by
    $C\, ||| \Lambda_{1} - \Lambda_{2} |||$, and
    \begin{align} \nonumber
	|G(\xi ; \omega) |
	& =  \Big|   \int_{\ERRE{n}} \thickspace
	             \rme^{-ix\cdot\xi} \thinspace F(0,x;\omega) \thinspace \mathrm{d}x \Big| \\ \nonumber
        & \le
	|| F(0,\cdot;\cdot)||_{L^{\infty}(\ERRE{n}_{x}\times S^{n-1})} \thinspace 
	\mathrm{Vol}(B_{n}(R)) \\ \label{estim:fourier:complement:cone}
	& \le
	    C R^{n} \thinspace |||\Lambda_{1} - \Lambda_{2} ||| 
    \end{align}
    shows that 
    $G$ is uniformly bounded in $\ERRE{n}_{\xi}\times S^{n-1}$.

   \begin{lemma}\label{lemma:4.3}
      Let $\Lambda_{1}$, $\Lambda_{2}$, 
      represent the Dirichlet to Neumann operators for the hyperbolic equations
      \eqref{hyp:eqn:no:scalar:pot}, and let $\alpha$ be as in corolary \ref{cor:4.2}.
     If $\alpha<2\pi$ and the divergence condition 
     \eqref{condition:divergence} holds, 
     then
    \begin{align} \label{eq:estim:Aj}
       \big| \widehat{A_{j}}(\tau,\xi) \big| 
          & \le 
          C\, |||\Lambda_{1} - \Lambda_{2} |||,\quad 0\le j \le n,
    \end{align}
    on the set 
    $\{(\tau,\xi) \thinspace : \thinspace |\tau| \le \frac{|\xi|}{2}\}$. 
   \end{lemma}
   \begin{proof}
    Proceeding as in the proof of theorem 3.3 in \cite{Salazar} 
    (see also \cite{Salazar2}), 
    for $(\tau,\xi)$ fixed with $|\tau|<\frac{1}{2}|\xi|$ we 
    can find unit vectors $\omega = \omega(\tau,\xi)$ parametrized by 
    an $(n-2)$-dimensional sphere with radius $r$, $\frac{\sqrt{3}}{2} \le r \le 1$,
    (we denote it by $rS^{n-2}$), satisfying
    $\tau + \omega(\tau,\xi)\cdot\xi = 0$, 
    as well as $\omega(\theta\tau,\theta\xi) =  \omega(\tau,\xi)$  
    for $\theta > 0$. In other words, we can find $\omega(\tau,\xi)$ homogenous
    of degree $0$ in $(\tau,\xi)$, such that 
    $(\tau,\xi) \perp \big(1,\omega(\tau,\xi)\big)$.
    If $n\ge 3$, we consider a maximal one dimensional sphere with radious $r$
    contained in $rS^{n-2}$ and choose unit vectors
    $\omega^{(1)}(\tau,\xi),\dots,\omega^{(n)}(\tau,\xi)$ 
    forming the vertices of a regular polygon with $n$ sides. If $n=2$
    we let $\omega^{(1)}(\tau,\xi)$ and $\omega^{(2)}(\tau,\xi)$ be the
    only two elements of $rS^{0}$. In both cases we then
    study the set of $n+1$ equations
    \begin{equation}\label{system:inhomo}
    \left\{
       \begin{aligned}
       \widehat{A_{0}}(\tau,\xi) + 
          \sum_{j=1}^{n} \omega_{j}^{(k)}(\tau,\xi)\widehat{A_{j}}(\tau,\xi) 
	  & = G\big(\xi;\omega^{(k)}(\tau,\xi)\big) ,
	  \quad\quad k =1,\dots,n \\
       \frac{1} {\sqrt{\tau^{2}+|\xi|^{2}}} \Big( \tau \widehat{A_{0}}(\tau,\xi) 
       + 
          \sum_{j=1}^{n} 
       \xi_{j}\widehat{A_{j}}(\tau,\xi) \Big) & = 0,
       \end{aligned}
    \right.
    \end{equation}
    where the last equation is a simple consequence of the divergence 
    condition \eqref{condition:divergence}.
    The left hand side of \eqref{system:inhomo} can be expressed as
    $ M(\tau,\xi)\hat{\mathcal{A}}(\tau,\xi)$,
    where 
    \begin{equation*}
    M(\tau,\xi) = 
       \begin{pmatrix} 
           1 & \omega_{1}^{(1)}(\tau,\xi) & \dots & \omega_{n}^{(1)}(\tau,\xi) \\
           1 & \omega_{1}^{(2)}(\tau,\xi) & \dots & \omega_{n}^{(2)}(\tau,\xi) \\
           \hdotsfor[2]{4}\\
           1 & \omega_{1}^{(n)}(\tau,\xi) & \dots & \omega_{n}^{(n)}(\tau,\xi) \\
           \frac{\tau}{\sqrt{\tau^{2}+|\xi|^{2}}} & 
           \frac{\xi_{1}}{\sqrt{\tau^{2}+|\xi|^{2}}} 
	   & \dots & 
           \frac{\xi_{n}}{\sqrt{\tau^{2}+|\xi|^{2}}} 
       \end{pmatrix}
    \end{equation*}
    has homogeneous entries of degree $0$ in $(\tau,\xi)$. We claim that 
    $M(\tau,\xi)$ is invertible. To prove this statement it suffices to show 
    that the homogeneous system
    \begin{equation}\label{system:homo}
    \left\{
       \begin{aligned}
       \widehat{A_{0}}(\tau,\xi) + 
          \sum_{j=1}^{n} \omega_{j}^{(k)}(\tau,\xi)\widehat{A_{j}}(\tau,\xi) 
	  & = 0, 
	  \quad\quad k =1,\dots,n \\
       \frac{1} {\sqrt{\tau^{2}+|\xi|^{2}}} \Big( \tau \widehat{A_{0}}(\tau,\xi) 
       + 
          \sum_{j=1}^{n} 
       \xi_{j}\widehat{A_{j}}(\tau,\xi) \Big) & = 0,
       \end{aligned}
    \right.
    \end{equation}
    has no non-trivial solution. By theorem 3.4 in \cite{Salazar} 
    (see also \cite{Salazar2}), 
    potentials satisfying the first $n$ equations are those of the form
    \begin{align*}
    \widehat{A_{0}}(\tau,\xi) = & \tau\, \Phi(\tau,\xi), \\
    \widehat{A_{j}}(\tau,\xi) = & \xi_{j}\, \Phi(\tau,\xi),\quad 1\le j \le n,
    \end{align*}
    for some smooth function $\Phi$. The last equation in 
    \eqref{system:homo}
    gives $\Phi(\tau,\xi)\sqrt{\tau^{2} + |\xi|^{2}} = 0$, which in turn
    leads to $\Phi \equiv 0$, and $\widehat{\mathcal{A}} = 0$. 

    Since $M(\tau,\xi)$ is invertible we can write
    \begin{equation*}
       \widehat{A_{j}}(\tau,\xi)  = 
          \sum_{k=1}^{n} c_{k,j}(\tau,\xi) G\big(\xi;\omega^{(k)}(\tau,\xi)\big),\quad
	   1\le k \le n,\thickspace 0 \le j \le n,
    \end{equation*}
    for some $c_{k,j}(\tau,\xi)$ homogeneous of degree $0$ in $(\tau,\xi)$. 
    It follows then that 
    \begin{align}\nonumber
       \big| \widehat{A_{j}}(\tau,\xi) \big| 
          & \le 
          \sum_{k=1}^{n} \big| c_{k,j}(\tau,\xi)  \big| 
	                 \big| G\big(\xi;\omega^{(k)}(\tau,\xi)\big) \big| \\ \label{non:uniform}
          & \le 
          C\, |||\Lambda_{1} - \Lambda_{2} ||| \sum_{k=1}^{n} \big| c_{k,j}(\tau,\xi)  \big|, 
    \end{align}
    where in the last line of the previous inequality we used the uniform 
    bound (\ref{estim:fourier:complement:cone}).

    In view of the homogeneity of the functions $c_{k,j}(\tau,\xi)$
    it suffices to work on the compact set
    $\{(\tau,\xi) \thinspace : \thinspace \tau^{2}+|\xi|^{2}=1,
    \thickspace|\tau| \le \frac{|\xi|}{2}\}$.
    The entries of the inverse matrix of $M(\tau,\xi)$ have
    the form
    \begin{equation*}
       c_{k,j}(\tau,\xi) = \frac{1}{\det M(\tau,\xi)} \mathrm{C}_{j,k}(\tau,\xi)
    \end{equation*}
    where $\mathrm{C}_{j,k}(\tau,\xi)$ is the $(j,k)$-cofactor of $M(\tau,\xi)$.
    Since the 
    entries of $M(\tau,\xi)$ have absolute value less or equal to one, 
    and since $\mathrm{C}_{j,k}(\tau,\xi)$ consists of sums of products of 
    $n$ such entries, we have
    \begin{equation*}
       | c_{k,j}(\tau,\xi) | 
       \le \frac{| \mathrm{C}_{j,k}(\tau,\xi) |}{| \det M(\tau,\xi) |} 
       \le \frac{n}{| \det M(\tau,\xi) |}.
    \end{equation*}

    The quantity $|\det M(\tau,\xi)|$ represents the $(n+1)$-dimensional
    volume generated by the vectors 
    $\{(1,\omega^{(1)}(\tau,\xi)),\dots,(1,\omega^{(n)}(\tau,\xi)),(\tau,\xi) \}$.
    Due to our choice of 
    $ \omega^{(1)}(\tau,\xi),\dots, \omega^{(n)}(\tau,\xi)$ 
    this volume does not depend on the point $(\tau,\xi)$.
    Moreover, 
    $|\det M(\tau,\xi)| = V \times \mathrm{P}(\tau,\xi)$ 
    where $\mathrm{P}(\tau,\xi)$ is the projection of $(\tau,\xi)$
    into the linear subspace generated by the set of vectors
    $\{(1,\omega^{(1)}(\tau,\xi)),\dots,(1,\omega^{(n)}(\tau,\xi))\}$ 
    and $V$ is the $n$-dimensional volume generated by these vectors.
    This projection is given by $C\sin \varphi$ where $\varphi$
    is the angle between $(\tau,\xi)$ and said subspace. Since
    the vectors $(1,\omega^{(k)}(\tau,\xi))$, $1 \le k \le n$, 
    are located in the boundary of the light cone   
    $\{(\tau,\xi) \thinspace : \thinspace |\tau| \ge |\xi| \}$,
    this angle is bounded below by $\frac{\pi}{8}$. Therefore
    the value
    $|\det M(\tau,\xi)|$ is uniformly bounded from below by
    $\mathrm{V} \sin \frac{\pi}{8}$ on 
    $\{(\tau,\xi) \thinspace : \thinspace \tau^{2}+|\xi|^{2}=1,
    \thickspace|\tau| \le \frac{|\xi|}{2}\}$.
    Hence
    \begin{equation*}
       | c_{k,j}(\tau,\xi) | 
       \le \frac{n}{V \sin \frac{\pi}{8}},
    \end{equation*}
    and by (\ref{non:uniform}) we obtain the uniform estimate
    \begin{align*} 
       \big| \widehat{A_{j}}(\tau,\xi) \big| 
          & \le C \, |||\Lambda_{1} - \Lambda_{2} ||| 
    \end{align*}
    on the set 
    $\{(\tau,\xi) \thinspace : \thinspace |\tau| \le \frac{|\xi|}{2}\}$. 
   \end{proof}
    The following statement is a result about harmonic measures,
    its proof can be found in \cite{Begmatov}.
    \begin{lemma}\label{lemma:4.4}
    Consider the strip 
    \begin{equation*}
    S = \{ z = z_{1} + iz_{2} \thinspace: \thinspace z_{1}\in\ERRE{}, 
           |z_{2}|< 2|\tau_{0}|\pi, \tau_{0}\neq 0\}
    \end{equation*}
    and the rays
    \begin{equation*}
     p_{1} = \{ z \thinspace: \thinspace -\infty < z_{1} \le -2|\tau_{0} |, z_{2}=0\}, \qquad
     p_{2} = \{ z \thinspace: \thinspace 2|\tau_{0}| \le z_{1} < \infty, z_{2}=0\} 
    \end{equation*}
    in the complex plane $\mathbb{C}$. \\
    If $E = p_{1} \cup p_{2}$ and $G=S\setminus E$ is the 
    strip with cuts along the rays $p_{1}$ and $p_{2}$, we have
    \begin{equation}\label{estim:harmon:measure}
       \frac{2}{3} < \varpi(z,E,G) \le 1,
    \end{equation}
    where $\varpi(z,G,E)$ is the harmonic measure of $E$ with respect to $G$. 
    More precisely
        \begin{equation}\label{eqn:Poisson:int}
           \varpi(\zeta) = 
                         \frac{1}{\pi} \int_{-\infty}^{\infty}
                         \chi_{E'}(t)\frac{\zeta_{2}}{(t-\zeta_{1})^{2}+\zeta_{2}^{2}}
                         \thinspace\mathrm{d}t, 
        \end{equation}
    where $\chi_{E'}(t)$ is the characteristic function of the set 
    $E' = \{ t \in \mathbb{R} \, : \, |t| \le 1 \} \cup \{ t \in \mathbb{R} \, : \, |t| > \rme \}$.
    \end{lemma}
     
    We now perform a rotation in $\xi$ space
    to make any given vector $(\tau,\xi)=(\tau,\xi_{1},\dots,\xi_{n-1},\xi_{n})$ have 
    the representation $(\tau,0,\dots,0,\nu)$. Based on the 
    previous statements we want to `embed' the $\nu$-axis into a strip in the 
    complex plane and use the bounds developed in the previous lemma.
   \begin{lemma}\label{lemma:4.5}
      Let $\Lambda_{1}$, $\Lambda_{2}$, 
      represent the Dirichlet to Neumann operators for the hyperbolic equations
      \eqref{hyp:eqn:no:scalar:pot}, and let $\alpha$ be as in corolary \ref{cor:4.2}.
     If $\alpha<2\pi$ and the divergence condition 
     \eqref{condition:divergence} holds,
     then on the set 
    $\{(\tau,\xi) \thinspace : \thinspace |\tau| > \frac{|\xi|}{2}\}$ we have 
    \begin{equation} \label{estim:fourier:cone}
    |\widehat{A_{j}}(\tau,\xi)|
        \le C \, \frac{ \rme^{ \frac{2|\tau|a}{3} } 
	|||\Lambda_{1}-\Lambda_{2}|||^{\frac{2}{3} }  }
	{ |\tau|^{\frac{1}{3}} },
    \end{equation}
    where $a$ is some positive number bigger than 
     the diameter of $\Omega$. 
   \end{lemma}
   \begin{proof}

    Since the potentials $A_{j}$,
    $0\le j \le n$, are 
    compactly supported, the functions
    $  \widehat{A_{j}} (\tau_{0},0,\dots,0,\nu) $ 
    admit an analytic extension in $\nu$ into the complex plane. Letting 
    \begin{gather*}
    \Pi = \{ \nu = (\nu_{1},\nu_{2}) \thinspace: \thinspace \nu_{1}\in\ERRE{}, 
           |\nu_{2}|< 2|\tau_{0}|\pi, \thinspace \tau_{0} \neq 0 \}, \\
     q_{1} = \{ \nu = (\nu_{1},\nu_{2}) \thinspace: \thinspace -\infty < \nu_{1} \le -2|\tau_{0} |, \nu_{2}=0\}, \\
     q_{2} = \{ \nu = (\nu_{1},\nu_{2}) \thinspace: \thinspace 2|\tau_{0}| \le \nu_{1} < \infty, \nu_{2}=0\} 
    \end{gather*}
    and restricting the potentials to the $\nu$-axis, 
    (\ref{estim:harmon:measure}) leads to
    \begin{equation*}
       \frac{2}{3} < \varpi(\nu,E_{1},G_{1}) \le 1,
    \end{equation*}
    where $E_{1} = q_{1} \cup q_{2}$ and $G_{1}=\Pi\setminus E_{1}$.
    Denoting by  $ v_{j}(\nu) = \widehat{A_{j}}
       (2\tau_{0},0,\dots,0,\nu), $
    the above restriction 
    we have by the two-constant theorem (see \cite{Krantz} Theorem 9.4.5)
    \begin{equation}\label{two:constant:ineq}
        |v_{j}(\nu) | \le m_{j}^{\frac{2}{3}} M_{j}^{\frac{1}{3}} 
    \end{equation}
    where $m_{j}$ and $M_{j}$ are the respective upper bounds of the modulus of
    $v(\nu)$ on the rays $q_{1}$ and $q_{2}$ and on the union of the lines
    $q^{\prime}_{1}=\{ (\nu_{1},\nu_{2}) \thinspace : \thinspace \nu_{1}\in\ERRE{}, \nu_{2} = -2|\tau_{0}|\pi \}$
    and
    $q^{\prime}_{2}=\{ (\nu_{1},\nu_{2}) \thinspace : \thinspace \nu_{1}\in\ERRE{}, \nu_{2} =  2|\tau_{0}|\pi \}$.
    We point out that the rays $q_{1}$ and $q_{2}$ are contained in the set 
    $\{(\tau,\xi) \thinspace:\thinspace |\tau | \le \frac{|\xi|}{2} \}$ 
    and that \eqref{eq:estim:Aj} provides and estimate 
    for $|v_{j}(\nu)|$ in that region. 
    To compute $M_{j}$ we resort to the equality
    \begin{equation*}
       v_{j}(\nu) =  \frac{1}{2\pi} \int_{\ERRE{}} 
              \rme^{-i (\nu_{1}+i\nu_{2}) x_{n}} 
	      W_{j}(2\tau_{0},0,\dots,0,x_{n}) \thinspace \mathrm{d}x_{n}
    \end{equation*}
    where $W_{j}$ is the Fourier transform of $A_{j}$ in all variables except $x_{n}$.
    These functions are compactly supported in $x_{n}$ and the above
    integrand is nonzero only on a bounded subset of the real numbers. Hence
    on $q^{\prime}_{1} \cup q^{\prime}_{2}$
    \begin{equation*}
       |v_{j}(\nu) | \le \frac{1}{2\pi} \mathrm{sup}_{x_{n}\in(-a(\Omega),a(\Omega))} 
                          | W_{j}(2\tau_{0},0,\dots,0,x_{n}) |
                     \int_{-\tilde{a}(\Omega)}^{\tilde{a}(\Omega)} 
		     \rme^{2|\tau_{0}| \pi x_{n}} \mathrm{d}x_{n},
    \end{equation*}
    where $\tilde{a}$ is a positive number bigger than $\mathrm{diam}(\Omega)$.
    Integration in $x_{n}$ then leads to
    \begin{equation*}
       |v_{j}(\nu) | \le C \frac{\rme^{2|\tau_{0}|a}}{|\tau_{0}|}
    \end{equation*}
    where $\nu \in q^{\prime}_{1} \cup q^{\prime}_{2}$ and $a=\tilde{a}\pi$.
    Therefore, when 
    $\nu$ is a real number satisfying $-2|\tau_{0}| < \nu < 2|\tau_{0}|$ we have by
    (\ref{two:constant:ineq})
    \begin{equation*}
       |v_{j}(\nu) | \le C \frac{\rme^{\frac{2|\tau_{0}|a}{3} } 
                         m_{j}^{\frac{2}{3} }}{|\tau_{0}|^{\frac{1}{3}} }.
    \end{equation*}
    The above arguments work for any line contained in the hyperplane 
    $\tau = \tau_{0}$ that passes through the origin. 
    Hence by (\ref{eq:estim:Aj}), for 
    $\{ |\tau| > \frac{|\xi}{2} \}$ we have
    \begin{equation*} 
    |\widehat{A_{j}}(\tau,\xi)|
        \le C\, \frac{ \rme^{ \frac{2|\tau|a}{3} } 
	|||\Lambda_{1}-\Lambda_{2}|||^{\frac{2}{3} }  }
	{ |\tau|^{\frac{1}{3}} }.
    \end{equation*}
   \end{proof}
     We can now
     establish the desired stability estimate for the vector potentials. 
     The general idea is to use the inequality 
    $||f||_{L^{\infty}} \le C\,|| \widehat{f} \thinspace ||_{L^{1}}$ 
    and partition $\ERRE{}_{\tau}\times\ERRE{n}_{\xi}$ in 
    an appropriate way.
    \begin{thm}\label{stability:theorem}
       Suppose that the vector and scalar potentials 
       $\mathcal{A}^{(l)}=(A^{(l)}_{0},\dots,A^{(l)}_{n})$, 
       $V^{(l)}$, $l=1,2,$ 
       are real valued, compactly supported and 
       $C^{\infty}$ in $t$ and $x$. 
       Let $ \mathcal{A} = (A_{0},A_{1},\dots,A_{n})$
       where $ A_{j} = A_{j}^{(1)} - A_{j}^{(2)} $ 
       and suppose that the following divergence condition holds
       \begin{equation*}
          \mathrm{div}\thinspace\mathcal{A} = 
	     \partial_{t}A_{0}(t,x) + 
	     \sum_{j=1}^n{} \partial_{x_{j}}A_{j}(t,x) = 0,
       \end{equation*}
       and that the entries of the vector potential satisfy
      \begin{equation*}
	  \mathrm{sup} 
	     \Big| \int_{-\infty}^{\infty} \big( A_{0} + \sum_{j=0}^{n} \omega_{j} A_{j} \big) 
	           (t+s,x+s\omega) \thinspace\mathrm{d}s \Big| < 2\pi, 
       \end{equation*}
       where the supremum is taken over $(t,x;\omega) \in [T_{1},T_{2}]\times\Omega\times S^{n-1}$.

       If $\Lambda_{l}$ represents the Dirichlet to Neumann
       operator associated to the hyperbolic problem (1)-(4), 
       then the stability estimate 
    \begin{equation}\label{eqn:four:stars}
       \max_{0\le j \le n}\Big|\Big| 
       A^{(1)}_{j}(t,x) - A^{(2)}_{j}(t,x) 
       \Big|\Big|_{L^{\infty}(\ERRE{}_{t}\times\ERRE{n}_{x})}
       \le C\,\Bigg[ \log 
             \frac{ 1 }
	          { ||| \Lambda_{1} - \Lambda_{2} |||  }
       \Bigg]^{-1}
    \end{equation}
    holds for $\Lambda_{1}$, $\Lambda_{2}$ satisfying $|||\Lambda_{1} - \Lambda_{2} ||| << 1$.
    \end{thm}
    \begin{proof}
    Let $\alpha$ be as in corolary \ref{cor:4.2}. Since $\alpha < 2\pi$,
    from the Fourier inversion formula we have
    \begin{equation}\label{eqn:star}
       A_{j}(t,x) = 
          \frac{1}{(2\pi)^{n+1}} \iint_{\ERRE{}_{\tau}\times\ERRE{n}_{\xi}} 
          \rme^{i(t\tau + x\cdot\xi)} \widehat{A_{j}}(\tau,\xi)
	  \thinspace\mathrm{d}\tau\mathrm{d}\xi.
    \end{equation}
    Taking absolute values we have for $\rho>0$
    \begin{align*}
       \big| A_{j}(t,x) \big|
       & \le \thinspace \frac{1}{(2\pi)^{n+1}} \iint_{\ERRE{}_{\tau}\times\ERRE{n}_{\xi}} 
            \big| \thinspace 
	    \widehat{A_{j}}(\tau,\xi)
	    \thinspace\big|
	    \thinspace\mathrm{d}\tau\mathrm{d}\xi \\
       & \le \thinspace \frac{1}{(2\pi)^{n+1}} \iint_{B(\rho_{1})} 
            \big| \thinspace 
	    \widehat{A_{j}}(\tau,\xi)
	    \thinspace\big|
	    \thinspace\mathrm{d}\tau\mathrm{d}\xi \\
       & \phantom{\le} + \thinspace \frac{1}{(2\pi)^{n+1}} \iint_{B(\rho_{1})^{c}} 
            \big| \thinspace 
	    \widehat{A_{j}}(\tau,\xi)
	    \thinspace\big|
	    \thinspace\mathrm{d}\tau\mathrm{d}\xi \\
      & = I_{1} + I_{2},
    \end{align*}
    where $B(\rho)$ denotes the $(n+1)$-dimensional ball
    $B(\rho)= \{ (\tau,\xi) \thinspace : \thinspace 
    |\tau|^{2} + |\xi|^{2} \le \rho^{2}\}$.
    Since for $0\le j \le n$, the potentials $A_{j}$, 
    are $C_{0}^{\infty}$ in $t$ and $x$, for 
    any $\beta > 0$, $\rho_{1}>0$, 
    if $|\tau|^{2} + |\xi|^{2} \ge \rho_{1}^{2}$ we have
    \begin{equation*}
       \big| 
       \widehat{A_{j}}(\tau,\xi)
       \big| 
       \le \frac{C}
       { (|\tau|^{2} + |\xi|^{2} )^{\frac{\beta}{2}} },
    \end{equation*}
    where $C$ depends on the derivatives of $A_{j}(t,x)$ up 
    to order $\beta$. When $\beta > n+1$, the integral $I_{2}$ converges. 
    Moreover, when $\beta > n+2$ and $\rho > 1 $, the following estimate holds
    \begin{equation}\label{estability:one:over:rho}
       I_{2} = \iint_{B(\rho)^{c}} 
       \big| 
       \widehat{A_{j}}(\tau,\xi)
       \big| 
          \mathrm{d}\tau\mathrm{d}\xi
       \le \frac{C}
       { \rho^{\beta - n - 1} }
       \le \frac{C}
       { \rho }.
    \end{equation}
    To estimate $I_{1}$ we break up the ball $B(\rho)$ into two smaller pieces
    \begin{equation*}
     \mathcal{C}_{1} = 
     B(\rho) \cap \Big\{ (\tau,\xi) \thinspace: \thinspace |\tau| < \frac{|\xi|}{2} \Big\}
     \quad \text{and} \quad
     \mathcal{C}_{2} = 
     B(\rho) \cap \Big\{ (\tau,\xi) \thinspace: \thinspace |\tau| \ge \frac{|\xi|}{2} \Big\}.
    \end{equation*}
    Then 
    \begin{equation*}
       I_{1} \le 
       \iint_{\mathcal{C}_{1}} 
       \big| 
          \widehat{A_{j}}(\tau,\xi)
       \big| 
          \mathrm{d}\tau\mathrm{d}\xi +
       \iint_{\mathcal{C}_{2}} 
       \big| 
          \widehat{A_{j}}(\tau,\xi)
       \big| 
          \mathrm{d}\tau\mathrm{d}\xi,
    \end{equation*}
    and since $\mathcal{C}_{1}$ is a subset of $B(\rho)$ we have
    \begin{equation*}
       I_{1} \le C \rho^{n+1} ||| \Lambda_{1} - \Lambda_{2} ||| +
       \iint_{C_{2}} 
       \big| 
          \widehat{A_{j}}(\tau,\xi)
       \big| 
          \mathrm{d}\tau\mathrm{d}\xi.
    \end{equation*}
    With this decomposition, $\mathcal{C}_{2}$ is contained in the
    set $\{(\tau,\xi) \thinspace:\thinspace |\tau | > \frac{|\xi|}{2} \}$. 
    Thus by (\ref{estim:fourier:cone})
    \begin{equation}\label{estability:two:constant}
       I_{2} \le C \rho^{n+1} ||| \Lambda_{1} - \Lambda_{2} ||| +
       C'\thinspace 
             \rme^{\frac{2\rho a}{3}} 
             ||| \Lambda_{1} - \Lambda_{2} |||^{\frac{2}{3}} \rho^{n+\frac{2}{3}}.
    \end{equation}
    Equations (\ref{eqn:star})-(\ref{estability:two:constant}) 
    lead to
    \begin{align}\label{eqn:two:starsone} 
       \big| 
          A_{j}(t,x)
       \big| 
       & \le \thinspace C\, \Big[
             \frac{1}{\rho} + \rho^{n+1} ||| \Lambda_{1} - \Lambda_{2} ||| + 
             \rho^{n+\frac{2}{3}} \thinspace \rme^{\frac{2\rho a}{3}} 
	            ||| \Lambda_{1} - \Lambda_{2} |||^{\frac{2}{3}} 
	     \Big] 
    \end{align}
    The rest of the proof is fairly standard. First we seek to impose a 
    condition on  $||| \Lambda_{1} - \Lambda_{2} |||$ so that the
    the third term in the right hand side of (\ref{eqn:two:starsone}) 
    dominates the second one. This can be done by simple minimization
    in $\rho$ of the function $\frac{\rme^{2\rho a}}{\rho}$ over the interval 
    $[1,+\infty)$. If $a<\frac{1}{2}$ we want 
    $||| \Lambda_{1} - \Lambda_{2} ||| < 2ae$ and if $a\ge\frac{1}{2}$ 
    then $||| \Lambda_{1} - \Lambda_{2} ||| < \rme^{2a}$. In both cases, 
    if $||| \Lambda_{1} - \Lambda_{2} ||| << 1$ then 
    \begin{align}\label{eqn:two:stars} 
       \big| 
          A_{j}(t,x)
       \big| 
       & \le \thinspace C\, \Big[
             \frac{1}{\rho} + 
                    \rho^{n+\frac{2}{3}} \thinspace \rme^{\frac{2\rho a}{3}} 
	            ||| \Lambda_{1} - \Lambda_{2} |||^{\frac{2}{3}} 
	     \Big].
    \end{align}
    The next step is to choose $\rho$ so that the two terms in the the right 
    hand side of (\ref{eqn:two:stars}) 
    are comparable. In other words we want $\rho$ to satisfy the identity
    \begin{equation*}
       \frac{C}{\rho} = 
              \rho^{n+\frac{2}{3}} \thinspace \rme^{\frac{2\rho a}{3}} 
	            ||| \Lambda_{1} - \Lambda_{2} |||^{\frac{2}{3}} 
    \end{equation*}
    for some constant $C$. Taking logarithms on both sides of the previous 
    equation yields the following equivalent identity
    \begin{equation}\label{eqn:three:stars}
       2 \log \frac{C} { \thinspace ||| \Lambda_{1} - \Lambda_{2} ||| } 
       = (3n+5) \log \rho + 2a\rho,
    \end{equation}
    where the right hand side of (\ref{eqn:three:stars}) is one to one when
    $\rho > 0 $ and hence it admits a unique solution.
    On the other hand, the inequality $\log\rho \le \rho$ for positive $\rho$ 
    as well as (\ref{eqn:three:stars}) lead to
    \begin{equation*}
       2 \log \frac{ C }
	         { \thinspace ||| \Lambda_{1} - \Lambda_{2} ||| } 
       \le (3n+5 + 2a)  \rho, 
    \end{equation*}
    or
    \begin{equation*}
       \frac{1}{\rho} \le 
             \frac{3n+5+2a}{2} \Bigg[ \log 
             \frac{ C }
	          { ||| \Lambda_{1} - \Lambda_{2} |||  }
       \Bigg]^{-1},
    \end{equation*}
    and (\ref{eqn:two:stars}) becomes
    \begin{equation*} 
       \big| 
          A_{j}(t,x)
       \big| 
       \le C''\Bigg[ \log 
             \frac{ C' }
	          { ||| \Lambda_{1} - \Lambda_{2} |||  }
       \Bigg]^{-1}
       \le C\Bigg[ \log 
             \frac{ 1 }
	          { ||| \Lambda_{1} - \Lambda_{2} |||  }
       \Bigg]^{-1},
    \end{equation*}
    where $C$ depends on $n$, $\Omega$ and derivatives of $A_{j}(t,x)$ for 
    $0\le j \le n$.
    \end{proof}

%% file: 03-Scalar.tex
\section{Stability of the scalar potentials}\label{new:section}
   In this section we establish a log-log type estimate for the scalar 
   potentials. We point out that the estimate from theorem 
   \ref{stability:theorem} is independent of the scalar potentials.
   This is because the term involving the difference of said potentials
   is not the leading term in the assympotics \eqref{eq:GF:RHS}
   and it does not survive the process of dividing by $k$ and taking the
   limit as $k\to +\infty$. In the following lines, we reuse the 
   techniques developed in the previous sections while following closely
   the ideas from Isakov and Sun in \cite{Isakov:Sun}.

    \begin{thm}\label{stability:theorem:2}
       Suppose that the vector and scalar potentials
       $\mathcal{A}^{(l)}=(A^{(l)}_{0},\dots,A^{(l)}_{n})$,
       $V^{(l)}$, $l=1,2,$
       are real valued, compactly supported and
       $C^{\infty}$ in $t$ and $x$.
       Let $V = V^{(1)} - V^{(2)}$, $ \mathcal{A} = (A_{0},A_{1},\dots,A_{n})$,
       where $ A_{j} = A_{j}^{(1)} - A_{j}^{(2)} $
       and suppose that the following divergence condition holds
       \begin{equation*}
          \mathrm{div}\thinspace\mathcal{A} = 
             \partial_{t}A_{0}(t,x) + 
             \sum_{j=1}^n{} \partial_{x_{j}}A_{j}(t,x) = 0,
       \end{equation*}
       and that the entries of the vector potential satisfy
      \begin{equation*}
          \mathrm{sup} 
             \Big| \int_{-\infty}^{\infty} \big( A_{0} + \sum_{j=0}^{n} \omega_{j} A_{j} \big) 
                   (t+s,x+s\omega) \thinspace\mathrm{d}s \Big| < 2\pi, 
       \end{equation*}
       where the supremum is taken over $(t,x;\omega) \in [T_{1},T_{2}]\times\Omega\times S^{n-1}$.

       If $\Lambda_{l}$ represents the Dirichlet to Neumann
       operator associated to the hyperbolic problem (1)-(4),
       then for $\Lambda_{1}$, $\Lambda_{2}$ satisfying 
       $|||\Lambda_{1} - \Lambda_{2} ||| << 1$, the following 
       stability estimates hold
    \begin{align*}
       \big|\big|\big| 
       \, \mathcal{A} \,
       \big|\big|\big|_{0}
       & \le C\,\Big( \log 
             \frac{ 1 }{|||\Lambda_{1} - \Lambda_{2} |||}
       \Big)^{-1} , \\
       \big|\big| \, V \,
       \big|\big|_{L^{\infty}(\ERRE{}_{t}\times\ERRE{n}_{x})}
       & \le C\,\Big( \log \big( \log 
             \frac{ 1 }{|||\Lambda_{1} - \Lambda_{2} |||} \big)
       \Big)^{-1} , 
    \end{align*}
    where 
   \begin{equation*}
      ||| \,\mathcal{A}\, |||_{0} =
      ||| \mathcal{A}^{(1)} - \mathcal{A}^{(2)} |||_{0} 
      := \max_{0\le j \le n} || 
         A^{(1)}_{j}(t,x) - A^{(2)}_{j}(t,x) 
         ||_{L^{\infty}(\ERRE{}_{t}\times\ERRE{n}_{x})}.
   \end{equation*}
    \end{thm}
    \begin{proof}
   In view of our previous results, it is enough to obtain a uniform 
   estimate for the X-ray transform along light rays of the difference 
   of the scalar potentials.
   By theorem \ref{stability:theorem},
   for arbitrary smooth compactly supported scalar potentials 
   $V^{(1)}\neq V^{(2)}$, we have 
   \begin{equation}\label{estimate:short}
      ||| \,\mathcal{A}\, |||_{0} 
      \le C \Big( \log \frac{1}{|||\Lambda_{1} - \Lambda_{2} |||} \Big)^{-1}.
   \end{equation}
   Green's formula \eqref{eqn:GF} with $u$ and $v$ solutions 
   of the forward and backward hyperbolic problem respectively, as well 
   as the triangle inequality give
   \begin{multline}\label{Greens:with:triangle}
      \big| \ELEDOTT{(V^{(1)}-V^{(2)})u}{v} \big| \le 
      \big| \ELEDOT{(\Lambda_{1}-\Lambda_{2})u}{v}_{\TT{}\times\partial\Omega} \big| \\
          + \sum_{j=0}^{n} \big| \ELEDOTT{(A^{(1)}_{j}-A^{(2)}_{j})u}{(-\rmi\DXJ{j} v)} \big| \\
          + \sum_{j=0}^{n} \big| \ELEDOTT{(A^{(1)}_{j}-A^{(2)}_{j})(-\rmi\DXJ{j} u)}{v} \big| \\ 
          + \sum_{j=0}^{n} \big| \ELEDOTT{\big[ (A_{j}^{(2)} )^{2}
          - (A_{j}^{(1)} )^{2}\big]u}{v} \big| .
   \end{multline}
   When $u$ and $v$ are given by the GO anzats developed in section 
   \ref{sec:review:prev:article}, the 
   discussion of the assymptotics of the derivatives 
   $\partial_{x_{j}}u,\partial_{x_{j}}v$, $0\le j \le n$, 
   give the estimates
   \begin{equation*}
      \big| \ELEDOT{(\Lambda_{1}-\Lambda_{2})u}{v}_{\TT{}\times\partial\Omega} \big|
      \le C k \big( |||\Lambda_{1}-\Lambda_{2}||| + \mathcal{O}(k^{-1}) \big), 
   \end{equation*}
   \begin{equation*}
      \big| \ELEDOTT{(A^{(1)}_{j}-A^{(2)}_{j})u}{(-\rmi\DXJ{j} v)} \big| 
      \le C k \big( ||| \,\mathcal{A}\, |||_{0} + \mathcal{O}(k^{-1}) \big), 
   \end{equation*}
   \begin{equation*}
      \big| \ELEDOTT{(A^{(1)}_{j}-A^{(2)}_{j})(-\rmi\DXJ{j} u)}{v} \big|  
      \le C k \big( ||| \,\mathcal{A}\, |||_{0} + \mathcal{O}(k^{-1}) \big), 
   \end{equation*}
   \begin{equation*}
      \big| \ELEDOTT{\big[ (A_{j}^{(2)} )^{2} - (A_{j}^{(1)} )^{2}\big]u}{v} \big| 
      \le C ||| \,\mathcal{A}\, |||_{0} , 
   \end{equation*}
   where the last inequality follows from the fact that 
   $|A^{(1)}_{j}(t,x)-A^{(2)}_{j}(t,x)| \le C$ for all
   $(t,x)\in\ERRE{}_{t}\times\ERRE{n}_{x}$.

   On the other hand, 
   the LHS of \eqref{Greens:with:triangle} gives
   \begin{multline}\label{LHS:Greens:with:triangle}
        \Big| \int_{T_{1}}^{T_{2}}\int_{\Omega} V(t,x)u(t,x)\overline{v(t,x)} \, \rmd x \rmd t \Big|
        = \Big| \int_{T_{1}}^{T_{2}}\int_{\Omega} 
       V(t,x) \chi_{1}(t,x)\chi_{2}(t,x) 
              \thinspace \times \\
              \rme^{-i \int_{-\infty}^{\frac{1}{2} ( t+\omega\cdot x )} 
                     \big( A_{0} + \sum_{j=1}^{n} \omega_{j}A_{j}\big)(t'+s,x'+s\omega) \thinspace \mathrm{d}s } 
                     \thinspace \mathrm{d}x\thinspace \mathrm{d}t + \cdots \Big|
   \end{multline}
   where $(t',x')$ is the projection of $(t,x)$ onto $\Pi_{(1,\omega)}$ and 
   ``$\cdots$'' represents terms of order $\mathcal{O}(k^{-1})$. Also,
   a simple analysis of $\rme^{iz}$ for small $|z|$ gives
   \begin{equation*}
      \rme^{-i \int_{-\infty}^{\frac{1}{2} ( t+\omega\cdot x )} 
               \big( A_{0} + \sum_{j=1}^{n} \omega_{j}A_{j}\big)(t'+s,x'+s\omega) \thinspace \mathrm{d}s } 
      = 1 + \mathcal{O}\big(||| \,\mathcal{A}\, |||_{0}\big)
   \end{equation*}
   and thus \eqref{Greens:with:triangle} leads to
   \begin{multline*}
      \Big| \int_{T_{1}}^{T_{2}}\int_{\Omega} 
       V(t,x) \chi_{1}(t,x)\chi_{2}(t,x) 
       \, \rmd x \rmd t \Big| \le \\
      C_{1}k \big(||| \,\mathcal{A}\, |||_{0} + |||\Lambda_{1}-\Lambda_{2}||| \big) 
      + C_{2} ||| \,\mathcal{A}\, |||_{0} + \mathcal{O}(k^{-1})
   \end{multline*}
   As in previous cases, the fact that the functions $\chi_{j}$ are supported
   near light rays shows that for $k>0$ the following estimate holds
   \begin{multline}\label{Greens:improved:for:scalar}
      \Big| \int_{-\infty}^{\infty} V(t+s,x+s\omega) \, \rmd s \Big| \le \\
      C_{1}k \big(||| \,\mathcal{A}\, |||_{0} + |||\Lambda_{1}-\Lambda_{2}||| \big) 
      + C_{2} ||| \,\mathcal{A}\, |||_{0} + \frac{C_{3}}{k}.
   \end{multline}
   Next we choose $k$ so that the first and last terms in the previous equation
   are comparable in size. To this end, let 
   \begin{equation*}
   k  = \big( |||\Lambda_{1} - \Lambda_{2} ||| + 
              ||| \,\mathcal{A}\, |||_{0} \big)^{-\frac{1}{2}},
   \end{equation*}
   then
   \begin{equation*}
      \big( |||\Lambda_{1} - \Lambda_{2} ||| + 
              ||| \,\mathcal{A}\, |||_{0} \big)k 
      = \frac{1}{k} \le C \big( |||\Lambda_{1} - \Lambda_{2} ||| + 
              ||| \,\mathcal{A}\, |||_{0} \big)^{\frac{1}{2}},
   \end{equation*}
   and \eqref{Greens:improved:for:scalar} gives
   \begin{align*}
      \Big| \int_{-\infty}^{\infty} V(t+s,x+s\omega) \, \rmd s \Big| 
       & \le C_{1} \big( |||\Lambda_{1}-\Lambda_{2}||| 
	     + ||| \,\mathcal{A}\, |||_{0} \big)^{\frac{1}{2}} 
             + C_{2} ||| \,\mathcal{A}\, |||_{0} \\
       & \le C \big( ||| \,\mathcal{A}\, |||^{\frac{1}{2}}_{0} 
              + |||\Lambda_{1}-\Lambda_{2}|||^{\frac{1}{2}} \big),
   \end{align*}
   where the last inequality holds when both 
   $|||\Lambda_{1} - \Lambda_{2}|||,||| \,\mathcal{A}\, |||_{0} < 1$
   (recall that if $0<\epsilon<1$, then $\epsilon < \sqrt{\epsilon} < 1$).
   Estimate \eqref{estimate:short} then gives
   \begin{align}\label{estimate:Xray:scalar:potentials} \nonumber
      \Big| \int_{-\infty}^{\infty} V(t+s,x+s\omega)\,\rmd s \Big| 
          & \le C \Big[ \big( \log\frac{1}{|||\Lambda_{1}-\Lambda_{2}|||} \big)^{-\frac{1}{2}} 
              + |||\Lambda_{1} - \Lambda_{2} |||^{\frac{1}{2}} \Big] \\  
          & \le C \big( \log\frac{1}{|||\Lambda_{1}-\Lambda_{2}|||} \big)^{-\frac{1}{2}} .
   \end{align}
   As in section \ref{sec:stability:estimates} we get from \eqref{estimate:Xray:scalar:potentials} 
   \begin{align*}
      \big|\big| \,V\, \big|\big|_{L^{\infty}(\ERRE{}_{t}\times\ERRE{n}_{x})} 
       & \le C' \Big( \log \frac{1}{ C \big( \log 
                \frac{1}{|||\Lambda_{1}-\Lambda_{2}|||} \big)^{-1/2}} \Big)^{-1} \\
       & \le C \Big( \log \big( \log \frac{1}{|||\Lambda_{1}-\Lambda_{2}|||}\big) \Big)^{-1}.
   \end{align*}
    \end{proof}

%% file: 10-References.tex

%% file: Stability-KleinGordon-Rev4.bbl
\begin{thebibliography}{99}
     \bibitem{Alessandrini} Alessandrini, G. \emph{Stable determination of conductivity by boundary 
                   measurements} Appl. Anal. 27 (1988) 153--72.
     \bibitem{Begmatov} Begmatov, Akram Kh. \emph{A certain inversion problem for the ray transform
                   with incomplete data}. Siberian Mathematical Journal, Vol. 42, No. 3, pp. 428--434, 2001.
     \bibitem{Belishev} Belishev, M. Kurylev, Y. \emph{To the reconstruction of a Riemannian manifold
                    via its spectral data (BC-method)}. Comm. Partial Differential Equations 17 (1992),
                    no. 5-6, 767-804.
     \bibitem{Choulli} Choulli, M. \emph{Une introduction aux probl\`emes inverses elliptiques
                   et paraboliques}. Math\'ematiques \& Applications. Vol. 65. Springer-Verlag, Berlin, 2009.
     \bibitem{Chavita} Cannon, J. R. Perez-Esteva, S \emph{An inverse problem for the heat equation} 
                   Inverse Problems 2 (1986) 395-403.
     \bibitem{Chavita2} Cannon, J. R. Perez-Esteva, S \emph{A note on an inverse problem related to the 3-D heat equation} 
                   Inverse Problems (Oberwolfach, 1986) 133-137. Internat. Schriftenreihe Numer. Math. 77,
                   Birkh\"auser, Basel, 1986.
     \bibitem{DosSantos:Uhlmann} Dos Santos Ferreira, D. Kenig, C. E. Sj\"ostrand, J. Uhlmann, G.
                   \emph{Determining a magnetic Schrödinger operator from partial Cauchy data}. 
                   Comm. Math. Phys. 271 (2007), no. 2, 467–-88. 
     \bibitem{DosSantos:Kenig:Salo} Dos Santos Ferreira, D. Kenig, C. E. Salo, M. Uhlmann, G. 
                   \emph{Limiting Carleman weights and anisotropic inverse problems}. Invent. Math. 178 (2009), 
                   no. 1, 119–-71.
     \bibitem{Eskin:approach} Eskin, Gregory. \emph{A new approach to hyperbolic inverse problems}.  
                   Inverse Problems,  22  (2006),  no. 3, 815--831.
     \bibitem{Eskin:approach2} Eskin, Gregory. \emph{A new approach to hyperbolic inverse problems II. 
                   Global step}. Inverse Problems,  23  (2007),  no. 6, 2343--2356.
     \bibitem{Eskin:time:dependent} Eskin, Gregory. \emph{Inverse hyperbolic problems with
                   time-dependent coefficients}, Comm. Partial Differential Equations, Vol. 32,
                   No. 11, pp. 1737--1758, 2007.
     \bibitem{Hartog} H\"ormander, Lars. \emph{An introduction to complex analysis in several variables}.
                   D. Van Nostrand Co. Princeton, N.J.-Toronto, Ont.-London, 1966.
     \bibitem{Hormander} H\"ormander, Lars. \emph{The analysis of Linear Partial Differential 
                   Operators I}. Springer, 1985.
     \bibitem{Isakov:first} Isakov, V. \emph{An inverse hyperbolic problem with many boundary measurements}.
                   Comm. Partial Differential Equations 16 (1991), no. 6-7, 1183–-95.
     \bibitem{Isakov:Sun} Isakov, V., Sun, Z. Q. \emph{Stability estimates for hyperbolic inverse 
                   problems with local boundary data.} Inverse Problems 8 (1992), no. 2, 193--206.
     \bibitem{Isakov:uniq:stab} Isakov, Victor. \emph{Uniqueness and stability in multi-dimensional 
                   inverse problems.} Inverse Problems 9 (1993), no. 6, 579--621.
     \bibitem{Isakov}  Isakov, Victor. \emph{Inverse problems for partial differential equations},
                   Applied Mathematical Sciences 127, Springer-Verlag New York Inc., 1998.
     \bibitem{Krantz} Krantz, Steven G. \emph{Geometric Function Theory. Explorations in complex analysis}.
                   Cornerstones, Birkh\"auser Boston, 2006.
     \bibitem{Kurylev} Kurylev, Y. \emph{Multi-dimensional inverse boundary problems by BC-method: 
                   groups of transformations and uniqueness results}. Math. Comput. Modelling, 
		   18 (1993) 33--45.
     \bibitem{Kurylev:Lassas} Kurylev, Y., Lassas, M. \emph{Hyperbolic inverse problems with data on 
                   a part of the boundary}. AMS/1P Stud. Adv. Math., 16 (2000) 259--72.
     \bibitem{Montalto} Montalto, C. \emph{Stable determination of a simple metric, a covector field 
                   and a potential from the hyperbolic Dirichlet-to-Nuemann map}. arXiv:1205.6425.
     \bibitem{Ramm:Sjostrand} Ramm, A. G., Sj\"ostrand. J. \emph{An inverse problem of the wave
                   equation}. Math. Z. 206, 119--130 (1991).
     \bibitem{Salazar} Salazar, R. \emph{Determination of time-dependent coefficients for a hyperbolic
                   inverse problem}. Inverse Problems 29 (2013) 095015. 
     \bibitem{Salazar2} Salazar, R. \emph{Determination of time-dependent coefficients for a hyperbolic
                   inverse problem}. arXiv:1009.4003. 
     \bibitem{Salazar3} Salazar, R. \emph{Determination of time-dependent coefficients for a hyperbolic
                   inverse problem}. PhD thesis, University of California Los Angeles, 2010.
     \bibitem{Schiff} Schiff, L. I. \emph{Quantum Mechanics}, Third Edition, McGraw-Hill, New York, (1955).
     \bibitem{Stefanov} Stefanov, P. \emph{Uniqueness of the multidimensional inverse scattering
                   problem for time dependent potentials}. Math. Z. 201, 541--560 (1989).
     \bibitem{Stefanov:Uhlmann} Stefanov, P., Uhlmann, G., \emph{Stability estimates for the hyperbolic 
                   Dirichlet to Neumann map in anisotropic media}.  J. Funct. Anal., 154 (1998),  no. 2, 
		   330--358.
     \bibitem{Stefanov:Uhlmann:2} Stefanov, P., Uhlmann, G., \emph{Stable determination of generic simple 
                   metrics from the hyperbolic Dirichlet-to-Neumann map}. Int. Math. Res. Not. 2005, no. 17, 1047--61.
     \bibitem{Tataru} Tataru, D. \emph{Unique continuation for solutions to PDE}. Commun. Partial Diff. 
                   Eqns., 20 (1995) 855--84.
     \bibitem{Uhlmann} Uhlmann, Gunther. \emph{Inverse scattering in anisotropic media. Surveys
     		   on solution methods for inverse problems}, pp 235--251, Springer, Vienna, 2000.
\end{thebibliography}
